\documentclass[11pt]{amsart}
\usepackage[margin=1.2in]{geometry}
\usepackage[utf8]{inputenc}
\usepackage{amsmath,amssymb,amsthm}
\usepackage{mathtools,mathdots,nicefrac}
\usepackage[normalem]{ulem}
\usepackage[mathscr]{euscript}
\usepackage{xcolor}
\usepackage{bbm}
\usepackage[titletoc]{appendix}
\usepackage{scrextend}
\usepackage[normalem]{ulem}
\usepackage{refcount}
\usepackage{needspace}

\title{A statistical investigation of a divisor-sum function}

\author{Ivan Aidun}
\address{Department of Mathematics\\ University of Wisconsin \\ Madison, WI 53706}
\email{aidun@wisc.edu}

\author{Lola Thompson}
\address{Mathematical Institute\\ Universiteit Utrecht\\ 3584 CD Utrecht, The Netherlands}
\email{l.thompson@uu.nl}

\newcommand{\N}{\mathbb{N}}
\newcommand{\R}{\mathbb{R}}
\newcommand{\varep}{\varepsilon}

\newcommand{\den}{\mathrm{\mathbf{d}}}
\newcommand{\upden}{\overline{\den}}
\newcommand{\lowden}{\underline{\den}}
\newcommand{\fracpar}[2]{\paren{\frac{#1}{#2}}}

\DeclareMathOperator{\id}{id}

\DeclarePairedDelimiter\floor{\lfloor}{\rfloor}
\DeclarePairedDelimiter\abs{\lvert}{\rvert}
\DeclarePairedDelimiter\paren{(}{)}
\DeclarePairedDelimiter\set{\{}{\}}

\theoremstyle{definition}
\newtheorem{theorem}{Theorem}[section]
\newtheorem{definition}[theorem]{Definition}
\newtheorem{cor}[theorem]{Corollary}

\newtheorem{lemma}[theorem]{Lemma}

\makeatletter
\newtheorem*{rep@theorem}{\rep@title}
\newcommand{\newreptheorem}[2]{%
\newenvironment{rep#1}[1]{%
 \def\rep@title{#2 \ref{##1}}%
 \begin{rep@theorem}}%
 {\end{rep@theorem}}}
\makeatother
\newreptheorem{theorem}{Theorem}

\makeatletter
\let\oldabs\abs
\def\abs{\@ifstar{\oldabs}{\oldabs*}}
\let\oldparen\paren
\def\paren{\@ifstar{\oldparen}{\oldparen*}}
\let\oldfloor\floor
\def\floor{\@ifstar{\oldfloor}{\oldfloor*}}
\makeatother

\begin{document}

\begin{abstract} The sum of proper divisors function $s(n)$ has been studied for more than 2000 years. In this paper we study statistical properties of the related function $S_s(n) \coloneqq \sum_{d \mid n} s(d)$. This function arises from a generalization of the practical numbers. Although $S_s(n)$ is neither additive nor multiplicative, we prove that $S_s(n)/n$ has a continuous asymptotic distribution function, and that its values are dense in the interval $[0,\infty)$. We evaluate its mean and establish the existence of all higher moments. Moreover, if $\mu_k$ denotes the $k^{th}$ moment of $S_s(n)/n$, we show that $$\mu_k^{1/k} \sim e^{2\gamma} (\log k)^2$$ as $k \rightarrow \infty$, where $\gamma$ is the Euler-Mascheroni constant. \end{abstract}

\subjclass[2020]{Primary 11N64; Secondary 11N60, 11N37} 
\keywords{asymptotic distribution function, nonmultiplicative arithmetic functions, divisor sums, moments of arithmetic functions} 

\maketitle

\section{Introduction}\label{sec:intro}

When faced with an erratic function, it is natural to want to understand its behavior: what are its extreme values? How does it behave on average? How does it behave ``typically''? These types of questions have been studied for a panoply of arithmetic functions since the early part of the $20^{th}$ century. Going a step further, one can regard an arithmetic function $f$ as a random variable on the discrete interval of integers in $[1,N]$, endowed with the uniform distribution, and apply tools from probability theory in order to study these functions. 

Among the basic questions that one might ask about $f$ are whether it possesses an asymptotic distribution, how its values are distributed throughout its range, and how its moments grow. For additive and multiplicative functions, a powerful collection of classical tools can be used to address these questions. However, when an arithmetic function does not possess either structure, the situation is far less clear. In this paper we study a natural example lying just outside the multiplicative setting and show that, despite its nonmultiplicativity, its statistical behavior can be described in considerable detail. We call this new function $S_s(n)$, and define it as follows:

\[S_s(n) \coloneqq \sum_{d \mid n} s(d),\]  where $s(d)$ is the sum of proper divisors of a positive integer $d$. Although $S_s(n)$ is not multiplicative, the identity $$S_s(n) = S_\sigma(n) - \sigma(n)$$ expresses it as a difference of two multiplicative functions, where $S_\sigma = \sum_{d \mid n} \sigma(d).$ As a result, $S_s$ lies close enough to the multiplicative setting to retain considerable structure, while remaining outside the scope of many classical tools. In addition to providing an interesting test case for applying probabilistic methods in a non-multiplicative setting, this function also has connections to the practical numbers. In what follows, we provide background on $s(n)$ and the practical numbers, motivating the study of $S_s(n).$


\subsection{The function $s(n)$}

Let $s(n)$ denote the sum over all positive proper divisors of $n$, i.e.,
\[s(n) = \sum_{\substack{d \mid n\\ d < n}} d.\]
We may write $s(n) = \sigma(n) - n$, where $\sigma(n)$ is the usual sum-of-divisors function. Note that $s(n)$ is neither additive nor multiplicative.

The function $s(n)$ has an ancient history, going back to the Pythagoreans and their study of deficient, abundant, and perfect numbers. Pomerance \cite{Pom} goes so far as to refer to $s(n)$ as ``the first function''. Despite being studied for over $2000$ years, surprisingly little is known about $s(n)$ today. Davenport's \cite{Dav} classical work on the densities of deficient and abundant numbers provides an early example of the statistical study of this function.



\subsection{Motivation from $f$-practical numbers}

The \emph{practical numbers}, introduced by Srinivasan in \cite{Sri}, are positive integers $n$ such that every number between 1 and $n$ can be represented as a sum of distinct divisors of $n$. There is a long history of studying the distribution of the practical numbers, culminating in a proof by Weingartner \cite{Wein} that their counting function $P(x)$ satisfies $P(x) \sim cx/\log x$ for some positive constant $c$. An analogous notion, the $\varphi$-practical numbers, arises by replacing the divisors $d$ by the values $\varphi(d)$. These were studied extensively by the second author (see, for example, \cite{Tho} where they were first defined and \cite{PTW} where an asymptotic for the counting function was proven).

Motivated by the studies of practical and $
\varphi$-practical numbers, Schwab and the second author \cite{ST} generalized this construction to $f$-practical numbers for positive-integer-valued arithmetic functions $f$: a number $n$ is $f$-practical if every integer between $1$ and $$S_f(n) = \sum_{d \mid n} f(d)$$ can be written as a sum of $f(d)$'s, for distinct divisors $d$ of $n$. Notice that $S_f(n)$ is the largest number that could be written as a sum of $f(d)$ where the values of $d$ are distinct, so it is the natural upper bound for the interval where we can expect this property to hold. The original practical numbers and the $\varphi$-practical numbers correspond to the $f$-practical numbers for $f = \id$ and $f = \varphi$, respectively. 

\subsection{ Main results}



In the spirit of the classical work of Davenport \cite{Dav} on $n/\sigma(n)$ and Schoenberg \cite{Sch28,Sch36} on $\varphi(n)/n$, it is natural to consider whether the function $S_s(n)/n$ possesses a distribution function. We prove the following result in \S \ref{sec:cdf}.
\begin{reptheorem}{thm:cdf}
The function $S_s(n)/n$ has a continuous asymptotic distribution function.
\end{reptheorem}

Schoenberg \cite{Sch28} also proved that the function $\varphi(n)/n$ has image dense in the interval $[0,1]$. Analogous to this result, we prove the following.

\begin{reptheorem}{thm:densevalues}
The values of $S_s(n)/n$ are dense in the interval $[0,\infty)$.
\end{reptheorem}

We also compute mean values for $S_s(n)$ and $S_s(n)/n$, establish the existence of all moments, and determine their asymptotic growth. In particular, we prove:

\begin{reptheorem}{thm:moment est}
The moments $\mu_k$ exist and are finite.
\end{reptheorem}

Moreover, we determine the asymptotic growth of the moments.

\begin{reptheorem}{thm:momentasymp}
The moments $\mu_k$ satisfy $$\mu_k^{1/k} \sim e^{2\gamma} (\log k)^2$$ as $k \rightarrow \infty$, where $\gamma$ is the Euler-Mascheroni constant.
\end{reptheorem}

Taken together, these results show that $S_s(n)/n$, although nonmultiplicative, still exhibits a great deal of the statistical behavior that is familiar from classical multiplicative functions. However, the proofs exploit multiplicative structure in three rather different ways. The classical analytic approach cannot be directly applied to prove that $S_s(n)/n$ has a continuous distribution function, due to the fact that the distribution function of $\log \sigma(n)/n$ is purely singular. Instead, our distribution theorem exploits the representation $S_s = S_\sigma - \sigma$ together with a modern criterion of Lebowitz-Lockard and Pollack for linear combinations of multiplicative functions. The density theorem instead uses a recurrence for $S_s(Nq)$, while the moment asymptotic connects the large values of $S_s(n)/n$ with the classical function $n/\varphi(n).$



\setcounter{section}{1}
\section{Tools from probabilistic number theory}\label{sec:tools}

In this section, we introduce the definitions and tools from probabilistic number theory that will be used in our proofs in Sections \ref{sec:dense}, \ref{sec:cdf}, \ref{sec:mvm}, and \ref{sec:momentasymptotic}.

\subsection{ Definitions and Notation}

A central concept in probabilistic number theory is that of asymptotic density, which is a formalization of the intuitive notion of the probability that an integer belongs to a set.

\begin{definition}
We define the \emph{asymptotic density} (also called the \emph{natural density} or simply \emph{density}) of a subset $A \subset \N$ to be
\[\den A = \lim_{N \to \infty} \frac{\#\set{a \in A\colon a \leq N}}{N},\]
when the limit exists. Replacing the limit by $\limsup$ (resp.\ $\liminf$) yields the \emph{upper density} (resp.\ \emph{lower density}), which we denote $\upden$ (resp.\ $\lowden$).
\end{definition}

The asymptotic density can be seen as a limit of the probabilities $\mathbb{P}(n \in A)$ where $n$ is restricted to the interval $[1,N]$. As such, asymptotic density preserves many nice properties of usual probabilities, but it does not form a measure on $\N$. In particular, the sets possessing an asymptotic density are not closed under countable union.

In classical probability theory, given a real random variable $X$ following some distribution, the distribution function $F$ associated to that distribution is $F(x) = \mathbb{P}(X \leq x)$. Any function arising this way will be non-decreasing and right-continuous (i.e., $\lim_{x\to x_0^+} F(x) = F(x_0)$). Moreover, such a function will satisfy $\lim_{x \to -\infty} F(x) = 0$ and $\lim_{x \to \infty} F(x) = 1$. We use these properties to define a general distribution function.

\begin{definition}
A non-decreasing function $F$ is a distribution function (d.f.) if $F$ is right-continuous and satisfies $\lim_{x \to -\infty} F(x) = 0$ and $\lim_{x \to \infty} F(x) = 1$.
\end{definition}

For our purposes, a ``random variable'' will be an arithmetic function $f$. If the function $f$ is well-behaved, then the function which appears will be a true distribution function according to the above definition.

\begin{definition}
Given an arithmetic function $f$, we define the sequence of functions
\[F_N(x) = \frac{\#\set{n \leq N \colon f(n) \leq x}}{N}.\]
We say $f$ has asymptotic distribution function (a.d.f.) $F$ if the functions $F_N$ converge pointwise to a function $F$, and if $F$ is a distribution function.
\end{definition} 

We note that if $f$ has an a.d.f.\ $F$, then $F(x) = \den\set{n: f(n)\leq x}$.

\begin{definition}
For an arithmetic function $f$, we define the \emph{mean value of $f$ over $n \leq x$}, for $x$ some positive real number, to be
\[M_x(f) = \frac{1}{x} \sum_{n \leq x} f(n).\]
Furthermore, we define the \emph{mean value of $f$} to be $M(f) = \displaystyle\lim_{x\to \infty} M_x(f)$ when the limit exists.

Similarly, the \emph{$k$th moment} of an arithmetic function $f$ is defined to be
\[\lim_{x \to \infty} \frac{1}{x} \sum_{n \leq x} f(n)^k,\]
when the limit exists.
\end{definition}

\subsection{Theorem of Erd\H{o}s-Wintner}

Because of the utility of a.d.f.s, a rich theory has been established on the subject of when certain arithmetic functions have an a.d.f.\ A powerful theorem in this vein is the Erd\H{o}s-Wintner Theorem \cite[p.475]{ten15}, which completely answers the question of the existence of an a.d.f.\ in the case of additive arithmetic functions.

\begin{theorem}[Erd\H{o}s-Wintner, 1939] \label{thm:erdoswintner}
Fix any real number $R > 0$. A real additive function $f(n)$ has a limiting distribution if and only if the following three series converge simultaneously:

\begin{center}
\begin{tabular}{ccc}
    \hspace{8mm} (i) $\displaystyle{\sum_{\abs{f(p)} > R} \frac{1}{p}}$; \hspace{8mm} & \hspace{8mm} (ii) $\displaystyle{\sum_{\abs{f(p)} \leq R} \frac{f(p)^2}{p}}$; \hspace{8mm} & \hspace{8mm} (iii) $\displaystyle{\sum_{\abs{f(p)} \leq R} \frac{f(p)}{p}}$. \hspace{8mm}\\
\end{tabular}
\end{center}

In this case, all three sums converge for all $R > 0$. The limiting d.f.\ is either absolutely continuous, purely singular, or discrete. It is continuous if and only if
\[\sum_{f(p) \neq 0} \frac{1}{p} = \infty.\]
\end{theorem}

The Erd\H{o}s-Wintner theorem gives insight not only into additive functions, but also multiplicative functions. If $g$ is a strictly positive multiplicative function satisfying certain reasonable conditions\footnote{$g$ cannot be almost everywhere almost zero, i.e., it cannot be the case that for all $\varep > 0$, $\den\set{n \colon g(n) > \varep} = 0$. An example of a function failing this condition is $f(n) = 1/n$. See \cite[Theorem 4]{Babu}.}, then $g$ possesses an a.d.f.\ $\psi$ if and only if the additive function $\log g$ possesses an a.d.f.\ $\omega$. In this case, $\omega(x) = \psi(e^x)$.

Perhaps most surprising is that the Erd\H{o}s-Wintner Theorem does not require considering $f(p^{\alpha})$ for any $\alpha > 1$. In some sense, this tells us that if an additive function $f$ has an a.d.f., then, for almost all $n$, the value of $f(n)$ is almost determined by its value on the squarefree part of $n$.

As an application of the Erd\H{o}s-Wintner theorem, one can prove the classical theorem of Davenport \cite{Dav} that $n/\sigma(n)$ has a continuous distribution function. The same kind of argument can be applied to the functions $\varphi(n)/n$ and $n/S_{\sigma}(n)$ to show that they, too, have a.d.f.s.

Since $S_s(n)/n$ is not multiplicative, the function $\log S_s(n)/n$ is not additive, so Erd\H{o}s-Wintner cannot be applied directly. However, we will use the continuous distribution functions for $\sigma(n)/n$ and $S_{\sigma}(n)/n$ furnished by these theorems to show $S_s(n)/n$ has a continuous distribution function in Section \ref{sec:cdf}.

\section{$S_s(n)/n$ is dense in $\R^+$}\label{sec:dense}

In this section we will show that the values $S_s(n)/n$ are dense in $[0,\infty)$. First, we begin by recalling a classical result of Schoenberg \cite{Sch36}:

\begin{theorem}[Schoenberg]\label{thm:sigma dense}
The values $n/\sigma(n)$ are dense in $[0,1]$.
\end{theorem}

Since the function $S_s(n)/n$ is not multiplicative, the argument Schoenberg used to prove Theorem \ref{thm:sigma dense} will not work. However, we are able to extract a version of Schoenberg's Theorem for $s(n)/n$ by writing it in terms of the function $\sigma(n)/n$. Namely, since $s(n)/n = \sigma(n)/n - 1$, it follows from Theorem \ref{thm:sigma dense} that the values of $s(n)/n$ are dense in $[0,\infty)$. One might hope that there is a similar representation for $S_s(n)/n$. For example, we can write
\begin{align*}
    S_s(n) &= \sum_{d \mid n} s(d)\\
    &= \sum_{d \mid n} (\sigma(d) - d)\\
    &= \sum_{d \mid n} \sigma(d) - \sum_{d \mid n} d\\
    &= S_{\sigma}(n) - \sigma(n).
\end{align*}
Then $S_s(n)/n = S_{\sigma}(n)/n - \sigma(n)/n$. However, determining whether the values of $S_s(n)/n$ are dense in $[0,\infty)$ from this seems to require that we be able to simultaneously control the growth of $S_{\sigma}(n)/n$ and $\sigma(n)/n$, which seems difficult.

To circumvent these problems, we introduce another relationship involving $S_s$: if $a$ and $b$ are relatively prime integers, then $S_s(ab) = S_s(a)S_s(b) + \sigma(a)S_s(b) + \sigma(b)S_s(a)$. To see this, we write
\begin{align}
    S_s(ab) &= S_{\sigma}(ab) - \sigma(ab) \notag \\
    &= S_{\sigma}(a) S_{\sigma}(b) - \sigma(a) \sigma(b) \notag \\
    &= (S_s(a) + \sigma(a))(S_s(b) + \sigma(b)) - \sigma(a) \sigma(b) \notag \\
    &= S_s(a)S_s(b) + \sigma(a)S_s(b) + \sigma(b)S_s(a). \label{eq:S_striple}
\end{align}

Observe that, when $p$ is a prime, $S_s(p) = 1$. We will use this fact repeatedly throughout the remainder of this section. We now proceed with our result.

\needspace{3cm}
\begin{theorem}\label{thm:densevalues}
The values $S_s(n)/n$ are dense in $[0,\infty)$.
\end{theorem}
\begin{proof}
Let $x \in [0,\infty)$, and index the primes in increasing order $p_1,p_2,\dots$. If $x = 0$, then the sequence $S_s(p_i)/p_i = 1/p_i$ converges to $x$. Otherwise, $x > 0$. In this case, we break the result down into the following two claims.

\vspace{1em}
\noindent \textit{Claim 1:} if $N > 1$ is an integer such that $\frac{S_s(N)}{N} < x$ and so that every prime factor of $N$ is smaller than
\[B(N) \coloneqq \frac{S_{\sigma}(N)/N + S_s(N)/N}{x - S_s(N)/N},\]
then we can find a prime $q = q(N)$ with $B(N) < q < 2B(N)$, and so that
\[\frac{1}{2}\paren{x + \frac{S_s(N)}{N}} < \frac{S_s(Nq)}{Nq} < x.\]

\vspace{1em}
\noindent \textit{Claim 2:} for any $x > 0$, there is an $N$ satisfying the hypotheses of Claim 1.

\vspace{1em}
To see how the theorem follows from the claims, fix $x$ and let $N$ be an integer satisfying the hypotheses of Claim 1. Starting from $N_1 = N$, we construct a sequence $N_1,N_2,\dots$ by letting $N_{i+1} = N_i q(N_i)$. Then we have
\[0 < x - \frac{S_s(N_{i+1})}{N_{i+1}} < \frac{1}{2} \paren{x - \frac{S_s(N_i)}{N_i}},\]
so the sequence $S_s(N_i)/N_i$ converges to $x$. (Notice that $B(N_{i+1}) \geq 2B(N_i) > q(N_i)$, so if $N_i$ satisfies the hypotheses of Claim 1, then $N_{i+1}$ does as well, and we are indeed able to build this sequence.)

To establish Claim 1, first notice that if $N$ is an integer and $q$ is a prime not dividing $N$, then by \eqref{eq:S_striple},
\begin{align}
\frac{S_s(Nq)}{Nq} &= \frac{S_s(N) S_s(q)}{Nq} + \frac{\sigma(N)S_s(q)}{Nq} + \frac{\sigma(q) S_s(N)}{Nq} \notag \\
&= \frac{1}{q}\paren{\frac{S_s(N)}{N} + \frac{\sigma(N)}{N}} + \frac{q+1}{q} \frac{S_s(N)}{N} \notag \\
& = \frac{1}{q}\frac{S_{\sigma}(N)}{N} + \frac{q+1}{q} \frac{S_s(N)}{N} \notag \\
&= \frac{1}{q} \paren{\frac{S_{\sigma}(N)}{N} + \frac{S_s(N)}{N}} + \frac{S_s(N)}{N}. \label{eq:product_formula}
\end{align}
Now, given $N$ satisfying the hypotheses of Claim 1, we must have that $B(N) \geq 2$ (since all the prime factors of $N$ are less than $B(N)$), and so by Bertrand's Postulate we can find a prime $q$ in the interval $(B(N), 2B(N))$. Such a $q$ is coprime to $N$, so by the above computation
\[\frac{S_s(Nq)}{Nq} = \frac{1}{q} \paren{\frac{S_{\sigma}(N)}{N} + \frac{S_s(N)}{N}} + \frac{S_s(N)}{N}.\]
Using that $q > B(N)$, we obtain
\begin{align*}
    \frac{1}{q} \paren{\frac{S_{\sigma}(N)}{N} + \frac{S_s(N)}{N}} + \frac{S_s(N)}{N} &< \frac{1}{B(N)} \paren{\frac{S_{\sigma}(N)}{N} + \frac{S_s(N)}{N}} + \frac{S_s(N)}{N}\\
    &= \paren{x - \frac{S_s(N)}{N}} + \frac{S_s(N)}{N} = x.
\end{align*}
Further, using that $q < 2B(N)$, we obtain
\begin{align*}
    \frac{1}{q} \paren{\frac{S_{\sigma}(N)}{N} + \frac{S_s(N)}{N}} + \frac{S_s(N)}{N} &> \frac{1}{2B(N)} \paren{\frac{S_{\sigma}(N)}{N} + \frac{S_s(N)}{N}} + \frac{S_s(N)}{N}\\
    &= \frac{1}{2} \paren{x - \frac{S_s(N)}{N}} + \frac{S_s(N)}{N} = \frac{1}{2}\paren{x + \frac{S_s(N)}{N}},
\end{align*}
as desired.

Finally, we show that there exists an $N$ satisfying the hypotheses of Claim 1. Fix $x > 0$, and let $p_k$ be the least prime so that $S_s(p_k)/p_k < x$. (We can find such a prime since $S_s(p)/p = 1/p$.) For $m \geq k$, let $P_m = \prod_{i=k}^m p_i$. From \eqref{eq:product_formula}, since $S_{\sigma}(N) \geq S_s(N)$, we have that $\frac{S_s(Nq)}{Nq} \geq \frac{q+2}{q} \frac{S_s(N)}{N}$ whenever $q$ is a prime not dividing $N$. Applying this inequality repeatedly, we have that
\[S_s(P_m)/P_m \geq \frac{1}{p_k} \prod_{i=k+1}^m \frac{p_i+2}{p_i}.\]
Since the product $\prod_{i > k} \frac{p_i+2}{p_i}$ diverges, so too does $S_s(P_m)/P_m$, and so there exists an $m \geq k$ so that $S_s(P_m)/P_m < x \leq S_s(P_{m+1})/P_{m+1}$. Notice that
\begin{align*}
    x - \frac{S_s(P_m)}{P_m} &\leq \frac{S_s(P_{m+1})}{P_{m+1}} - \frac{S_s(P_m)}{P_m}\\
    &= \frac{S_s(p_{m+1}P_m)}{p_{m+1} P_m} - \frac{S_s(P_m)}{P_m}\\
    &= \frac{1}{p_{m+1}} \paren{\frac{S_{\sigma}(P_m)}{P_m} + \frac{S_s(P_m)}{P_m}}.
\end{align*}
So, $B(P_m) \geq p_{m+1}$ while every prime factor of $P_m$ is smaller than $p_{m+1}$, and so $P_m$ satisfies the hypotheses of Claim 1.
\end{proof}

\section{Continuous distribution function}\label{sec:cdf}

We now prove that $S_s(n)/n$ has a continuous distribution function. Note that our approach differs from the classical analytic approach (cf., \cite{Sch36}, \cite{Sch28}) for an important reason. Using that $S_s(n)/n = S_\sigma(n)/n - \sigma(n)/n$, it is tempting to observe that $\log \sigma(n)/n$ and $\log S_\sigma(n)/n$ are additive functions with continuous distribution functions, and then apply the Erd\H{o}s-Wintner Theorem to these distribution functions. However, it turns out that the distribution function for $\log \sigma(n)/n$ is purely singular, which makes it difficult to directly use these two distribution functions to create a distribution function for $S_s(n)/n$.

To get around this problem, we make use of modern technology that was recently introduced by Lebowitz-Lockard and Pollack \cite{L-LP}. If $f$ is a real-valued arithmetic function, we say \emph{f clusters around the real number $x$} if there exists a real number $d > 0$ such that for all $\varep > 0$,
\[\upden \set{n \colon x - \varep < f(n) < x + \varep} \geq d.\]
If $f$ does not cluster around any $x$, we say $f$ is \emph{nonclustering}. Suppose the arithmetic function $f$ has an a.d.f.\ $F$. It is easy to see that if $F$ is continuous then $f$ is nonclustering. Recall that when $F$ exists, it can be expressed as $\den \set{n : f(n) \leq x}$. Note that for any $\varep > 0$ we have
\[\upden \set{n \colon f(n) = x} \leq \den \set{n \colon x - \varep < f(n) \leq x + \varep} = F(x+\varep) - F(x - \varep).\]
Since $F$ is continuous, as $\varep \to 0$ the right-hand side goes to 0. Thus, $\den \set{n \colon f(n) = x} = 0$. Indeed, the converse also holds.

\begin{lemma}\label{lemma:cdf nonclustering}
If the arithmetic function $f$ has an a.d.f.\ $F$, and if $f$ is nonclustering, then $F$ is continuous.
\end{lemma}
\begin{proof}
Recall that $F$ is the pointwise limit of the ``partial'' distribution functions $F_N$ defined as
\[F_N(x) = \frac{\#\set{n \leq N \colon f(n) \leq x}}{N}.\]
Then, we have
\begin{align*}
    F(x+\varep) - F(x - \varep) &= \lim_{N\to \infty} F_N(x+\varep) - F_N(x- \varep)\\
    &= \den \set{n \colon x- \varep < f(n) \leq x+\varep}\\
    &\leq \upden \set{n \colon x - \varep < f(n) < x + \varep}.
\end{align*}
Thus, by the assumption that $f$ is nonclustering, as $\varep \to 0$, we have $F(x+ \varep) - F(x - \varep) \to 0$. Therefore, $F$ is continuous.
\end{proof}

We will use the following two theorems, which appear as Theorem 1 and Proposition 5 in \cite{L-LP}, respectively.

\begin{theorem}[Lebowitz-Lockard and Pollack] \label{thm:nonclustering}
Let $f_1,...,f_k$ be multiplicative arithmetic functions taking values in the nonzero real numbers and satisfying the following conditions:

\begin{enumerate}
    \item $f_1$ is nonclustering,
    \item none of $f_1,\dots,f_k$ cluster around 0,
    \item for all $i < j$ with $i, j \in \{1,2,...,k\}$, the function $f_i/f_j$ is nonclustering.
\end{enumerate}
Then for all nonzero $c_1,...,c_k \in \mathbb{R}$, the arithmetic function $F \coloneqq c_1 f_1 + \cdots c_k f_k$ is nonclustering. 
\end{theorem}

\begin{theorem}[Lebowitz-Lockard and Pollack] \label{thm:lin combo df}
Let $f_1, \dots, f_k$ be positive-valued multiplicative functions each possessing a distribution function. Then for any $c_1,\dots,c_k \in \R$, the function $c_1 f_1 + \dots + c_k f_k$ also has a distribution function.
\end{theorem}

Both of these theorems are proven by explicit estimation of upper densities by using the arithmetic properties of the functions $f_i$. We now proceed with the proof of Theorem \ref{thm:cdf}.

\begin{theorem}\label{thm:cdf}
The function $S_s(n)/n$ has a continuous a.d.f.
\end{theorem}
\begin{proof}
Recall that we can write 
\[S_s(n) = \sum_{d \mid n} (\sigma(d) - d) = S_{\sigma}(n) - \sigma(n).\]
Thus,
\[\frac{S_s(n)}{n} = \frac{S_\sigma(n)}{n} - \frac{\sigma(n)}{n}\]
is a difference of two multiplicative functions.

Let $f_1 = S_\sigma(n)/n$, $f_2 = \sigma(n)/n$, and $F = f_1 + (-1)f_2$. We have previously stated that $f_1$ and $f_2$ have distribution functions, so by Theorem \ref{thm:lin combo df} above, $F$ has an a.d.f.\ To show that the distribution function for $F$ is continuous, by Lemma \ref{lemma:cdf nonclustering} it suffices to show that it satisfies the hypotheses of Theorem \ref{thm:nonclustering}. We may apply Theorem \ref{thm:erdoswintner} to the additive functions $\log f_1$, $\log f_2$, and $\log (f_1/f_2)$ to show that $f_1$, $f_2$ and $f_1/f_2$ have continuous a.d.f.s. Thus, conditions (1)-(3) of Theorem \ref{thm:nonclustering} are satisfied. Therefore, $F$ is nonclustering. Since a distribution function for an arithmetic function $F$ is continuous precisely when $F$ is nonclustering, it follows that $F$ is continuous.
\end{proof}

\section{Mean values and moments of $S_s(n)/n$}\label{sec:mvm}

In this section we will look at some common statistics for the function $S_s(n)/n$. In the first subsection we will compute the mean values $M_x(S_s(n))$ and $M(S_s(n)/n)$. These results will ground our discussion in the following subsection, where we will establish the existence of moments of $S_s(n)/n$. 

\subsection{Mean values of $S_s(n)$ and $S_s(n)/n$}

To begin with, recall that by an elementary summation argument, $M_x(\sigma(n)) = \zeta(2)x/2 + O(\log x)$. We can use this fact to derive $M_x(S_{\sigma}(n))$ as follows:
\begin{align*}
    \frac{1}{x} \sum_{n \leq x} S_{\sigma}(n) &= \frac{1}{x} \sum_{n \leq x} \sum_{d \mid n} \sigma(d)\\
    &= \frac{1}{x} \sum_{\substack{d,q \\ dq \leq x}} \sigma(d)\\
    &= \frac{1}{x} \sum_{q \leq x} \sum_{d \leq x/q} \sigma(d)\\
    &= \frac{1}{x} \sum_{q \leq x} \paren{ \frac{\zeta(2)}{2} \paren{\frac{x}{q}}^2 + O\paren{\frac{x \log x}{q}}}.
    \end{align*}
 
From here it is a straightforward computation to verify that $M_x(S_{\sigma}(n)) = \zeta(2)^2x/2 + O((\log x)^2)$. We use these two values to compute the following result.

\begin{theorem}
The mean value $M_x(S_s(n))$ is given by
\[M_x(S_s(n)) = \frac{\zeta(2)(\zeta(2)-1)}{2}x + O((\log x)^2).\]
\end{theorem}
\begin{proof}
By linearity of $M_x$, we compute
\begin{align*}
    M_x(S_s(n)) &= M_x(S_{\sigma}(n) - \sigma(n))\\
    &= M_x(S_{\sigma}(n)) - M_x(\sigma(n))\\
    &= \frac{\zeta(2)(\zeta(2)-1)}{2}x + O((\log x)^2).
\end{align*}
\end{proof}

The following is an immediate corollary.

\begin{cor}
We have
\[M(S_s(n)/n) = \zeta(2)(\zeta(2)-1).\]
\end{cor}
\begin{proof}
Consider the sum $\sum_{n \leq x} S_s(n)/n$. Applying partial summation with $a_n = S_s(n)$ and $f(n) = 1/n$ we find
\begin{align*}
    \sum_{n \leq x} \frac{S_s(n)}{n} &= \frac{1}{x}\sum_{n \leq x} S_s(n) + \int_1^x \frac{\sum_{n \leq t} S_s(n)}{t^2}\ dt\\
    &= M_x(S_s(n)) + \int_1^x \frac{M_t(S_s(n))}{t}\ dt\\
    &= \frac{\zeta(2)(\zeta(2)-1)}{2}x + O((\log x)^2) + \int_1^x \paren{\frac{\zeta(2)(\zeta(2)-1)}{2} + O((\log t)^2/t)}\ dt\\
    &= \frac{\zeta(2)(\zeta(2)-1)}{2}x + O((\log x)^2) + \paren{\frac{\zeta(2)(\zeta(2)-1)}{2}t}\bigg\rvert_1^x + O((\log t)^3\big\rvert_1^x)\\
    &= \zeta(2)(\zeta(2)-1)x + O((\log x)^3).
\end{align*}
Thus, $M(S_s(n)/n) = \lim_{x \to \infty} \frac{1}{x} \sum_{n \leq x} S_s(n)/n = \zeta(2)(\zeta(2)-1)$.
\end{proof}

\subsection{Existence of the moments of $S_s(n)/n$}\label{sec:moments}

In this section, we establish the existence of moments of $S_s(n)/n$, i.e., the quantities
\[\mu_k = \lim_{n \to \infty} \frac{1}{n}\sum_{i=1}^n (S_s(i)/i)^k.\]
We will make use of a powerful tool known as Wintner's Mean Value Theorem for multiplicative functions \cite[Theorem 1, p. 138]{Post}.
\begin{theorem}[Wintner's Mean Value Theorem]\label{thm:wintner}
If $g$ is a multiplicative function satisfying
\renewcommand{\theenumi}{\roman{enumi}}
\begin{center}
\begin{minipage}{.4\textwidth}
\begin{enumerate}
    \item $\displaystyle{\sum_p \frac{\abs{g(p)-1}}{p} < \infty}$
    \item $\displaystyle{\sum_p \sum_{\nu=2}^{\infty} \frac{\abs{g(p^{\nu})-g(p^{\nu-1})}}{p^{\nu}} < \infty}$
\end{enumerate}
\end{minipage}
\end{center}
then the mean value of $g$ exists and is finite.
\end{theorem}

There are a few other facts we will make use of to establish our estimates for $\mu_k$. We will use the following expressions for the functions $\sigma$ and $S_{\sigma}$:
\begin{align}
    \sigma(p^{\nu}) &= p^{\nu}\paren{1+\frac{1}{p-1}} - \frac{1}{p-1}, \label{sigma} \\
    S_{\sigma}(p^{\nu}) &= p^{\nu}\paren{1+\frac{1}{p-1}}^2 - \frac{\nu+2}{p-1} - \frac{1}{(p-1)^2}. \label{S_sigma}
\end{align}
We obtain these expressions by writing
\begin{align*}
    \sigma(p^{\nu}) &= \sum_{i=0}^{\nu} p^i\\
    &= \frac{p^{\nu+1}-1}{p-1}.
\end{align*}
Pulling out $p^{\nu}$ yields \eqref{sigma}, and \eqref{S_sigma} follows from a similar argument.

We may now proceed with the result.

\begin{theorem} \label{thm:moment est}
The moments $\mu_k$ exist and are finite. 
\end{theorem}
\begin{proof}
First, the Binomial Theorem yields
\begin{align*}
    (S_s(i)/i)^k &= \frac{(S_{\sigma}(i) - \sigma(i))^k}{i^k}\\
    &= \frac{1}{i^k} \sum_{j=0}^k \binom{k}{j} (-1)^j (\sigma(i))^j (S_{\sigma}(i)^{k-j}).
\end{align*}
Each of the functions $h_{k,j}(i) = (\sigma(i))^j(S_{\sigma}(i))^{k-j}/i^k$ is multiplicative, and below we will use Wintner's Mean Value Theorem to show that each has finite mean. From the existence of mean values for the $h_{k,j}$, we conclude that the moments $\mu_k$ exist and are finite.

We first turn our attention to sum (i) in Theorem \ref{thm:wintner}. Since $n \leq \sigma(n) \leq S_{\sigma}(n)$ for all $n$, we have that $0 \leq h_{k,j}(p) - 1 \leq h_{k,0}(p) - 1$, and so it suffices to check that sum (i) converges for $g = h_{k,0}$. Using expression \eqref{S_sigma}, we get
\begin{align*}
    h_{k,0}(p) - 1 &\leq \fracpar{S_{\sigma}(p)}{p}^k - 1\\
    &< \paren{1+\frac{1}{p-1}}^{2k} - 1\\
    &= \frac{p^{2k} - (p-1)^{2k}}{(p-1)^{2k}}\\
    &= \frac{p^{2k} - (p^{2k} - 2k p^{2k-1} + \text{terms of lower degree})}{(p-1)^{2k}}\\
    &\ll_k \frac{p^{2k-1}}{(p-1)^{2k}}\\
    &\ll_k \frac{1}{p}.
\end{align*}
Thus, for $g = h_{k,0}$, the summands in (i) are $O(1/p^2)$, so the sum converges.

For the double sum (ii), we fix $k,j$ and use expressions \eqref{sigma} and \eqref{S_sigma}  to estimate
\begin{align*}
    h_{k,j}(p^{\nu}) &= \fracpar{\sigma(p^{\nu})}{p^{\nu}}^j \fracpar{S_{\sigma}(p^{\nu})}{p^{\nu}}^{k-j}\\
    &= \paren{\paren{1 + \frac{1}{p-1}}+ O\fracpar{1}{p^{\nu+1}}}^j \paren{\paren{1+\frac{1}{p-1}}^2 + O\fracpar{\nu}{p^{\nu+1}}}^{k-j}\\
    &= \paren{\paren{1 + \frac{1}{p-1}}^j + O_{k,j}\fracpar{1}{p^{\nu+1}}} \paren{\paren{1 + \frac{1}{p-1}}^{2k-2j} + O_{k,j}\fracpar{\nu}{p^{\nu+1}}}\\
    &= \paren{1+\frac{1}{p-1}}^{2k-j} + O_{k,j}\fracpar{\nu}{p^{\nu+1}}.
\end{align*}
Thus, the numerator of the inner sum (ii) is $\abs{h_{k,j}(p^\nu) - h_{k,j}(p^{\nu-1})} = O(\nu p^{-(\nu+1)})$. Therefore, the terms of the inner sum are $O(\nu p^{-(2\nu + 1)})$. We can evaluate the series $S = \sum_{\nu=2}^{\infty} \frac{\nu}{p^{2\nu + 1}}$ by using the geometric series $G = \sum_{\nu = 2}^{\infty} x^{(2\nu + 2)}$. We have $G = x^6/(1-x^2)$, so taking the derivative of both sides with respect to $x$ yields
\begin{align*}
    \frac{6x^5 - 4x^7}{(1-x^2)^2} &= \frac{d}{dx} \sum_{\nu=2}^{\infty} x^{2\nu+2}\\
    &= \sum_{\nu=2}^{\infty} (2\nu+2) x^{2\nu + 1}\\
    &= 2\paren{\sum_{\nu = 2}^{\infty} \nu x^{2\nu+1} + \sum_{\nu=2}^{\infty} x^{2\nu+1}}.
\end{align*}
Notice that the first term inside the parentheses becomes $S$ when evaluated at $x = 1/p$, and the second term is geometric. Rearranging and solving for $S$ gives us
\[S = \frac{2p^2 - 1}{(p^2-1)^2 p^3}.\]
So, we conclude that the inner sum converges to a value that is $O(p^{-5})$. Therefore, the double sum converges. Having checked that the hypotheses of Wintner's Mean Value Theorem hold, we conclude that each $h_{k,j}$ has a finite mean value. \end{proof}


\section{Asymptotic growth of the moments}\label{sec:momentasymptotic}


In this section, we determine the asymptotic growth of $\mu_k$. 

\begin{theorem} \label{thm:momentasymp}
The moments $\mu_k$ satisfy
$$\mu_k^{1/k} \sim e^{2\gamma}(\log k)^2$$ as $k \rightarrow \infty,$ where $\gamma$ is the Euler-Mascheroni constant.
\end{theorem}

We will require several lemmas before we proceed with the main proof.

\begin{lemma} For each positive integer $r$, define $$R(n):= \frac{n}{\varphi(n)}$$ and $$M_r:= \lim_{x \rightarrow \infty} \frac{1}{x} \sum_{n \leq x} R(n)^r.$$ Then the mean value $M_r$ exists and, as $r \rightarrow \infty,$ $$M_r^{1/r} \sim e^\gamma \log r,$$ where $\gamma$ denotes the Euler-Mascheroni constant.
\end{lemma}

\begin{proof}

    The mean value $M_r$ exists and is given by the Euler product $$M_r = \prod_p \left(1 + \frac{1}{p} \left(\left(1 - \frac{1}{p}\right)^{-r} -1\right)\right),$$ as recorded in Martin-Pollack-Smith \cite[see the proof of Proposition 4.3]{MPS}

    For the upper bound, we split the Euler product at $p = r.$ If $p \leq r$, we have $$1 + \frac{(1 - 1/p)^{-r} - 1}{p} \leq (1 - 1/p)^{-r}.$$ For $p > r$, $$(1 - 1/p)^{-r} - 1 \ll \frac{r}{p},$$ so $$\sum_{p > r} \log \left(1 + \frac{(1 - 1/p)^{-r} - 1}{p}\right) \ll r \sum_{p >r} \frac{1}{p^2} = O(1).$$ Thus, Mertens' Theorem gives us $$M_r^{1/r} \leq (1 + o(1))\prod_{p \leq r}\left(1 - \frac{1}{p}\right)^{-1} \sim e^\gamma \log r.$$ 

    We will define a parameter $z$ which depends on $r$, namely $$z = \frac{r}{(\log r)^2}.$$ Since every Euler factor is at least $1$, we can disregard primes $p > z.$ Moreover, each Euler factor satisfies $$1 + \frac{(1 - 1/p)^{-r} - 1}{p} \geq \frac{1}{p}(1 - 1/p)^{-r}.$$ As a result, we have $$M_r^{1/r} \geq \prod_{p \leq z} (1 - 1/p)^{-1} \exp\left(-\frac{1}{r} \sum_{p \leq z} \log p\right).$$ Observe that $$\sum_{p \leq z} \log p = O(z) = o(r),$$ and by Mertens' Theorem $$\prod_{p \leq z} (1 - 1/p)^{-1} \sim e^\gamma \log z \sim e^\gamma \log r.$$ Therefore, we have \begin{align}\label{eq:Mrasymp} M_r^{1/r} \sim e^\gamma \log r.\end{align} 
\end{proof}

\begin{lemma}\label{lem:limsup} For each positive integer $k$, let $$\mu_k := \lim_{x \rightarrow \infty} \frac{1}{x} \sum_{n \leq x} \left(\frac{S_s(n)}{n}\right)^k,$$ whose existence is guaranteed by Theorem \ref{thm:moment est}. Then $$\limsup_{k \rightarrow \infty} \frac{\mu_k^{1/k}}{(\log k)^2} \leq e^{2\gamma}.$$
\end{lemma}

\begin{proof}
Write $A(n) = \frac{S_\sigma(n)}{n}$ and $B(n) = \frac{\sigma(n)}{n}.$ Then $\frac{S_s(n)}{n} = A(n) - B(n).$ By \eqref{S_sigma}, we have $$\frac{S_\sigma(p^\nu)}{p^\nu} \leq \left(1 + \frac{1}{p-1}\right)^2 = \left( \frac{p^\nu}{\varphi(p^\nu)}\right)^2.$$ Thus, we have $$0 \leq \frac{S_s(n)}{n} \leq A(n) \leq R(n)^2 = \left(\frac{n}{\varphi(n)}\right)^2.$$ After taking $k^{th}$ powers and averaging over $n \leq x$, we obtain $$\mu_k \leq M_{2k}.$$ Taking $k^{th}$ roots yields $$\mu_k^{1/k} \leq M_{2k}^{1/k} = \left(M_{2k}^{1/(2k)}\right)^2.$$ Using the asymptotic \eqref{eq:Mrasymp} that we obtained above, we have $$M_{2k}^{1/(2k)} \sim e^\gamma \log(2k) \sim e^\gamma \log k.$$ Therefore, we obtain $$\limsup_{k \rightarrow \infty} \frac{\mu_k^{1/k}}{(\log k)^2} \leq e^{2\gamma}$$ as our upper bound.
\end{proof}

\begin{lemma}\label{lem:liminf} With $\mu_k$ defined as above, we have $$\liminf_{k \rightarrow \infty} \frac{\mu_k^{1/k}}{(\log k)^2} \geq e^{2\gamma}.$$
\end{lemma}

\begin{proof}
For the lower bound, let $$y = \frac{k}{(\log k)^3}.$$ For every prime $p \leq y$, define $$\nu_p := \left\lceil \frac{3 \log k}{\log p}\right\rceil,$$ so that $p^{\nu_p} \geq k^3$ and define $$m = m_k = \prod_{p \leq y} p^{\nu_p}.$$ Recall from \eqref{S_sigma} that \begin{align*} \frac{S_{\sigma}(p^{\nu_p})}{p^{\nu_p}} &= \left(1+\frac{1}{p-1}\right)^2 - \frac{\nu_p+1}{p^{\nu_p}(p-1)} - \frac{1}{p^{\nu_p-1}(p-1)^2}.\end{align*} Then we can write $$A(p^{\nu_p}) = (1 - 1/p)^{-2}(1 - \delta_p),$$ where $$\delta_p = \frac{(\nu_p + 1)(p-1)}{p^{\nu_p+2}} + \frac{1}{p^{\nu_p + 1}} \ll \frac{\nu_p + 1}{p^{\nu_p + 1}}.$$ Since $p^{\nu_p} \geq k^3,$ we have $$\sum_{p \leq y} \delta_p \ll \frac{\pi(y) \log k}{k^3} = o(1).$$ Hence, $$\prod_{p\leq y} (1 - \delta_p) = 1 + o(1).$$ Thus, we have $$A(m) \sim \prod_{p \leq y}(1 - 1/p)^{-2} \sim e^{2\gamma}(\log y)^2 \sim e^{2\gamma}(\log k)^2.$$ At the same time, we have $$B(m) \leq \prod_{p \leq y}(1-1/p)^{-1},$$ so that $$\frac{B(m)}{A(m)} \ll \frac{1}{\log y} = o(1).$$ Therefore, we can conclude that $$\frac{S_s(m)}{m} = A(m) - B(m) \sim e^{2\gamma}(\log k)^2.$$

Since a single integer $m$ contributes zero density, we now need to ``thicken'' the set. Let $\nu_p(n)$ denote the exponent of $p$ in the prime factorization of $n$, and consider the set $$\mathcal{E}_k = \{n: \nu_p(n) = \nu_p \ \mathrm{for \ every} \ p \leq y\}.$$ This set has natural density $$d(\mathcal{E}_k) = \prod_{p \leq y} \frac{1 - 1/p}{p^{\nu_p}} = \frac{1}{m} \prod_{p \leq y}(1 - 1/p).$$ Observe that $$\log m = \sum_{p \leq y} \nu_p \log p \ll \pi(y) \log k + y \ll \frac{k}{(\log k)^2}= o(k).$$ Then we have \begin{align*} \log d(\mathcal{E}_k) &= -\log m + \sum_{p \leq y} \log(1 - 1/p) \\ &= -\log m - \log \log y + O(1) \\ &= o(k).\end{align*} Thus, we obtain $$d(\mathcal{E}_k)^{1/k} = \exp \left(\frac{\log d(\mathcal{E}_k)}{k}\right) = 1 + o(1).$$ 

If $n \in \mathcal{E}_k$, write $n = mr,$ where all of the prime factors of $r$ are greater than $y.$ Since $A$ and $B$ are multiplicative, then $$C(n) := \frac{B(n)}{A(n)}$$ is also multiplicative. Note that $A(r) \geq 1$ and $0 < C(r) \leq 1.$ Therefore, we have \begin{align*}\frac{S_s(n)}{n} &= A(n)(1 - C(n))\ \\ &=A(m)A(r)(1 - C(m)C(r)) \\ &\geq A(m)(1 - C(m)) \\ &= \frac{S_s(m)}{m}.\end{align*}

Thus, we have $$\mu_k \geq d(\mathcal{E}_k)\left(\frac{S_s(m)}{m}\right)^k.$$ After taking $k^{th}$ roots, we obtain: $$\mu_k^{1/k} \geq d(\mathcal{E}_k)^{1/k} \frac{S_s(m)}{m} \sim e^{2 \gamma} (\log k)^2.$$ 


\end{proof}

\begin{proof}[Proof of Theorem \ref{thm:momentasymp}]
By Lemmas \ref{lem:limsup} and \ref{lem:liminf}, the $\limsup$ and $\liminf$ coincide, and $$\mu_k^{1/k} \sim e^{2\gamma}(\log k)^2.$$
    
\end{proof}

By taking the logarithm of both sides of the asymptotic expression, we obtain the following corollary:

\begin{cor}\label{cor:logmu_k} As $k \rightarrow \infty$, we have $$\log \mu_k = 2k \log \log k + 2 \gamma k + o(k).$$
\end{cor}


\section*{Acknowledgements} This project grew out of an honors thesis that the first author completed as an undergraduate student at Oberlin College, while working under the direction of the second author. The authors would like to thank Oberlin College for providing them with the opportunity to work together. In addition, they would like to thank Paul Pollack for helpful comments on an early draft of this manuscript, and Alisa Sedunova for suggesting the problem that they solve in Section \ref{sec:momentasymptotic}. The late stages in the preparation of this manuscript took place while the second author was on sabbatical at the Max Planck Institute for Mathematics and the Centre de Recherches Mathématiques. She would like to thank both institutions for providing her with a pleasant working environment. 

\bibliography{bibliography}
\bibliographystyle{plain}
\end{document}